\documentclass[a4paper,11pt]{article}
\usepackage{amsmath,amsthm,amssymb,enumitem,xcolor}
\usepackage[mathscr]{euscript}

\usepackage[hang]{footmisc}
\usepackage[nosort,nocompress,noadjust]{cite}

\usepackage[bookmarks=false,hyperfootnotes=false,colorlinks,
    linkcolor={red!60!black},
    citecolor={blue!50!black},
    urlcolor={blue!80!black}]{hyperref}

\renewcommand{\eqref}[1]{\hyperref[#1]{(\ref{#1})}}

\newlist{enumlist}{enumerate}{2}
\setlist[enumlist,1]{labelindent=0cm,label=\arabic*.,ref=\arabic*,labelwidth=2.5ex,labelsep=0.5ex,leftmargin=3ex,align=left,topsep=0.5ex,itemsep=1ex,parsep=1ex}
\setlist[enumlist,2]{labelindent=0cm,label=\theenumlisti.\arabic*.,ref=\arabic*,labelwidth=5ex,labelsep=0.5ex,leftmargin=5.5ex,align=left,topsep=0.5ex,itemsep=1ex,parsep=1ex}

\newlist{itemlist}{itemize}{1}
\setlist[itemlist]{labelindent=0cm,label=$\bullet$,labelwidth=2.5ex,labelsep=0.5ex,leftmargin=3ex,align=left,topsep=0.5ex,itemsep=1ex,parsep=1ex}

\numberwithin{equation}{section}

{\theoremstyle{definition}\newtheorem{definition}{Definition}[section]
\newtheorem*{definition*}{Definition}

\newtheorem{remark}[definition]{Remark}
\newtheorem{problem}[definition]{Open problem}

\newtheorem*{example*}{Example}
\newtheorem*{examples*}{Examples}}

\newtheorem{proposition}[definition]{Proposition}
\newtheorem{lemma}[definition]{Lemma}
\newtheorem{theorem}[definition]{Theorem}

\newtheorem{letterthm}{Theorem}

{\theoremstyle{definition}}

\renewcommand{\Re}{\operatorname{Re}}

\newcommand{\C}{\mathbb{C}}
\newcommand{\cC}{\mathcal{C}}
\newcommand{\eps}{\varepsilon}
\newcommand{\al}{\alpha}
\newcommand{\be}{\beta}
\newcommand{\ot}{\otimes}
\newcommand{\Z}{\mathbb{Z}}
\newcommand{\vphi}{\varphi}

\newcommand{\id}{\mathord{\text{\rm id}}}
\newcommand{\om}{\omega}
\newcommand{\N}{\mathbb{N}}
\newcommand{\ovt}{\mathbin{\overline{\otimes}}}

\newcommand{\si}{\sigma}
\newcommand{\R}{\mathbb{R}}

\newcommand{\cZ}{\mathcal{Z}}
\newcommand{\Ad}{\operatorname{Ad}}

\newcommand{\cF}{\mathcal{F}}

\newcommand{\actson}{\curvearrowright}

\newcommand{\cU}{\mathcal{U}}

\newcommand{\Aut}{\operatorname{Aut}}

\newcommand{\supp}{\operatorname{supp}}

\newcommand{\Out}{\operatorname{Out}}
\newcommand{\Inn}{\operatorname{Inn}}

\newcommand{\St}{\mathscr{S}}
\newcommand{\Serg}{\St_{\text{\rm erg}}}

\newcommand{\Erg}{\mathord{\text{\textit{Erg}}}}

\begin{document}
\begin{center}
{\boldmath\LARGE\bf Ergodic states on type III$_1$ factors\vspace{0.5ex}\\ and ergodic actions}

\bigskip

{\sc by Amine Marrakchi\footnote{UMPA, CNRS ENS de Lyon, Lyon (France). E-mail: amine.marrakchi@ens-lyon.fr} and Stefaan Vaes\footnote{\noindent KU~Leuven, Department of Mathematics, Leuven (Belgium). E-mail: stefaan.vaes@kuleuven.be.\\ S.V. is supported by FWO research project G090420N of the Research Foundation Flanders and by Methusalem grant METH/21/03 –- long term structural funding of the Flemish Government.}}
\end{center}

\begin{abstract}\noindent
Since the early days of Tomita-Takesaki theory, it is known that a von Neumann algebra $M$ that admits a state $\vphi$ with trivial centralizer $M_\vphi$ must be a type III$_1$ factor, but the converse remained open. We solve this problem and prove that such ergodic states form a dense $G_\delta$ set among all faithful normal states on any III$_1$ factor with separable predual. Through Connes' Radon-Nikodym cocycle theorem, this problem is related to the existence of ergodic cocycle perturbations for outer group actions, which we consider in the second part of the paper.
\end{abstract}

\section{Introduction}

The \emph{centralizer} $M_\vphi$ of a faithful normal state $\vphi$ on a von Neumann algebra $M$ is defined as the von Neumann subalgebra of elements $a \in M$ satisfying the trace-like property $ \vphi(ax) = \vphi(x a)$ for all $x \in M$. The centralizer $M_\vphi$ coincides with the fixed point subalgebra of the modular automorphism group $(\si_t^\vphi)_{t \in \R}$. The state $\vphi$ is a \emph{trace} when $M_\vphi=M$, or equivalently when $(\si_t^\vphi)_{t \in \R}$ is trivial. At the extreme opposite, the state $\vphi$ is said to be \emph{ergodic} if $M_\vphi = \C 1$, which amounts to requiring that the modular automorphism group $(\si_t^\vphi)_{t \in \R}$ acts ergodically on $M$.

It is rather difficult to produce ergodic states. In fact, the first example of a nontrivial factor with an ergodic state was constructed by using the CAR functor from quantum field theory, see the corollary of \cite[Theorem 1 in Section 3]{HT70}. After Connes established in \cite{Con72} his classification of type III factors into subtypes III$_\lambda, \: \lambda \in [0,1]$, Longo proved in \cite[Proof of Theorem 3]{Lon78} that if a von Neumann algebra $M$ admits an ergodic faithful normal state then it must be a type III$_1$ factor. However, the converse implication remained an open problem. Given the importance of type III$_1$ factors in quantum field theory (see e.g.\ \cite{Yng04}), this question continued to pop up in the literature. It also appeared explicitly in \cite{MU12}, where a natural Banach space geometry property for the predual of a von Neumann algebra $M$ is shown to be equivalent to the absence of ergodic states on corners $pMp$. Our first main result solves this question affirmatively.

For every von Neumann algebra $M$, we denote by $\St(M)$ the set of faithful normal states on $M$ and by $\Serg(M)$ the subset of ergodic faithful normal states. We denote by $\St_1(M)$ the Polish space of all normal states on $M$.

\begin{letterthm}\label{thmA}
Let $M \neq \C 1$ be a nontrivial von Neumann algebra with separable predual. Then the following statements are equivalent.
\begin{enumlist}
\item $M$ is a type III$_1$ factor.
\item $\Serg(M)$ is nonempty.
\item $\Serg(M)$ is a dense $G_\delta$ subset of $\St_1(M)$.
\end{enumlist}
\end{letterthm}

The proof of Theorem~\ref{thmA}  is based on a Baire category argument. In order to write $\Serg(M)$ as a countable intersection of dense open sets in $\St(M)$, we combine the Connes-St{\o}rmer transitivity theorem \cite[Theorem 4]{CS76} and Popa's ``local quantization'' theorem \cite[Theorem A.1.2]{Pop92}. As we explain in Remark \ref{rem.not-separable}, the separability assumption is essential: there are type III$_1$ factors $M$ with nonseparable predual that are countably decomposable (i.e.\ $\St(M)$ is nonempty) but for which $\Serg(M)$ is empty.

Given $\vphi,\psi \in \St(M)$, Connes' Radon-Nikodym theorem (see \cite[Th\'{e}or\`{e}me 1.2.1]{Con72}) provides a continuous $1$-cocycle $u_t = [D\psi : D\vphi]_t$ for $\sigma^\vphi$ satisfying $\sigma_t^\psi = \Ad u_t \circ \sigma_t^\vphi$ for all $t \in \R$. Theorem~\ref{thmA} is thus naturally connected to the question which group actions $G \actson^\al M$ on a von Neumann algebra admit an ergodic cocycle perturbation.

When $\Gamma$ is a countable amenable group and $M$ is a II$_1$ factor with separable predual, we prove the following general result. We denote by $\cC(\al)$ the space of $1$-cocycles, i.e.\ maps $v : \Gamma \to \cU(M) : g \mapsto v_g$ satisfying $v_{gh} = v_g \al_g(v_h)$ for all $g,h \in \Gamma$. Using the topology of pointwise $\|\cdot\|_2$-convergence, $\cC(\al)$ is a Polish space. For every $v \in \cC(\al)$, we denote by $\Ad v \circ \al$ the action defined by $(\Ad v \circ \al)_g = (\Ad v_g) \circ \al_g$ for all $g \in \Gamma$.

\begin{letterthm}\label{thmB}
Let $\Gamma$ be a countably infinite amenable group and $\Gamma \actson^\al M$ an outer action of $\Gamma$ on a II$_1$ factor $M$ with separable predual.

Then $\{v \in \cC(\al) \mid \Ad v \circ \al \;\;\text{is ergodic}\;\}$ is a dense $G_\delta$ subset of $\cC(\al)$. In particular, $\al$ admits a cocycle perturbation that is ergodic.
\end{letterthm}

In Theorem \ref{thm.ergodic-action-Z}, we deduce from Theorem~\ref{thmB} that an automorphism $\al$ of a II$_1$ factor $M$ admits an ergodic inner perturbation $\Ad u \circ \al$ if and only if all nonzero powers of $\al$ are outer. We also note in Remark \ref{rem.non-separable-actions} that the separability assumption in Theorem~\ref{thmB} is again essential.

In Section \ref{sec.ergodic-actions}, we introduce the class $\Erg$ of countable groups $\Gamma$ with the property that every outer action of $\Gamma$ on a II$_1$ factor $M$ admits an ergodic cocycle perturbation. By Theorem~\ref{thmB}, every infinite amenable group belongs to $\Erg$. It is rather easy to see (Proposition \ref{prop.free-product}) that $\Erg$ is closed under taking a free product with an arbitrary group, so that the free groups belong to $\Erg$. However, ``rigid'' groups do not belong to $\Erg$, see Proposition \ref{prop.no-go-thm}.

Finally, Theorems~\ref{thmA} and \ref{thmB} lead to several open questions and problems that we discuss in Section \ref{sec.problems}.

{\bf Acknowledgment.} We would like to thank Cyril Houdayer for thought provoking discussions on the existence problem of ergodic states. We thank the referee for their excellent remarks.

\section{\boldmath Ergodic states on type III$_1$ factors}

To prove Theorem~\ref{thmA}, we need several lemmas. Already in \cite[Proof of Theorem 3]{Lon78}, it was proven that a von Neumann algebra with an ergodic faithful normal state must be a type III$_1$ factor; see also \cite[Corollary 1.10.8]{Bau95}. For completeness we provide a self-contained argument, also proving that such a state is automatically weakly mixing.

Given $\vphi \in \St(M)$, we write $\|x\|_\vphi = \sqrt{\vphi(x^* x)}$ for all $x \in M$.

\begin{lemma}\label{lem.easy-implication}
Let $\vphi$ be a faithful normal state on a von Neumann algebra $M \neq \C 1$. If $M_\vphi = \C 1$, then $M$ is a type III$_1$ factor and the unitary representation $(\sigma_t^\vphi)_{t \in \R}$ on the orthogonal complement of $\C 1$ inside $L^2(M,\vphi)$ is weakly mixing.
\end{lemma}
\begin{proof}
Denote $H = L^2(M,\vphi)$ and let $\xi_\vphi \in H$ be the vector given by $1 \in M$. Denote by $U_t(x \xi_\vphi) = \sigma_t^\vphi(x) \xi_\vphi$ the unitary representation as in the formulation of the lemma.

We first prove that $U_t$ is weakly mixing on $\xi_\vphi^\perp$. For every $\lambda > 0$, denote by $H_\lambda \subset H$ the closed subspace of vectors $\eta \in \xi_\vphi^\perp$ satisfying $U_t(\eta) = \lambda^{it} \eta$ for all $t \in \R$. Denote by $P_\lambda$ the orthogonal projection of $H$ onto $H_\lambda$. We have to prove that $H_\lambda = \{0\}$ for all $\lambda > 0$. Assume the contrary. We then find $\lambda > 0$ so that $P_\lambda \neq 0$. We thus find $x \in M$ with $\vphi(x) = 0$ and $P_\lambda(x \xi_\vphi) \neq 0$. Define $y \in M$ as the unique element of minimal $\|\cdot\|_\vphi$ in the $\|\cdot\|_\vphi$-closed convex hull of
$$\bigl\{ \lambda^{-it} \si_t^\vphi(x) \bigm| t \in \R \bigr\} \; .$$
Since $P_\lambda(x \xi_\vphi) = y \xi_\vphi$, we have that $y \neq 0$. Also, $\vphi(y) = 0$ and $\sigma_t^\vphi(y) = \lambda^{it} y$ for all $t \in \R$. It follows that $y^* y$ and $y y^*$ belong to $M_\vphi = \C 1$, so that $y$ is a nonzero multiple of a unitary $u \in \cU(M)$ satisfying $\vphi(u) = 0$ and $\sigma_t^\vphi(u) = \lambda^{it} u$ for all $t \in \R$. In particular, the element $u$ is analytic under $(\sigma_t^\vphi)_{t \in \R}$, so that
$$1 = \vphi(u^* u) = \vphi(\si_i^\vphi(u) u^*) = \lambda^{-1} \vphi(uu^*) = \lambda^{-1} \; .$$
We have proven that $\lambda = 1$. This means that $u \in M_\vphi = \C 1$. But $\vphi(u) = 0$ and we have reached a contradiction. So, $U_t$ is weakly mixing on $\xi_\vphi^\perp$.

Since $\cZ(M) \subset M_\vphi$, we have that $M$ is a factor. To see that $M$ is of type III$_1$, we consider the continuous core $N = M \rtimes_{\sigma^\vphi} \R$ with its subalgebra $L(\R)$ and prove that $N$ is a factor. Recall that we may identify $N$ with the subalgebra of $\theta$-fixed points in $M \ovt B(L^2(\R))$, where $\theta_t = \sigma_t^\vphi \otimes \Ad \lambda_{-t}$. The relative commutant of $1 \ot L(\R)$ inside $M \ovt B(L^2(\R))$ equals $M \ovt L(\R)$. On $M \ovt L(\R)$, we have that $\theta_t = \sigma_t^\vphi \ot \id$. Since $M_\vphi = \C 1$, it follows that $L(\R)' \cap N = L(\R)$.

Define the closed subgroup $\Lambda \subset \R$ of $t \in \R$ such that $\si_t^\vphi = \id$. Since $L(\R)' \cap N = L(\R)$, the center of $N$ equals $L(\Lambda) \subset L(\R)$. So if $\Lambda = \{0\}$, it follows that $\cZ(N) = \C 1$ so that $M$ is of type III$_1$. If $\Lambda = \R$, it follows that $\vphi$ is a trace, so that $M = M_\vphi = \C 1$, which we excluded. When $\Lambda$ is a nontrivial closed subgroup of $\R$, the unitary representation $U_t$ is periodic, which contradicts its weak mixing on $\xi_\vphi^\perp$ proven above.
\end{proof}

We will need Popa's ``local quantization'' theorem.
 \begin{theorem}[{\cite[Theorem A.1.2]{Pop92}}] \label{lem.local-quantization}
Let $N$ be a II$_1$ factor and $x \in N$. For every $\varepsilon > 0$, there exists a finite dimensional abelian subalgebra $A \subset N$ such that $\| E_{A' \cap N}(x)-\tau(x)1\|_2 \leq \varepsilon$.
\end{theorem}

For our proof, we also need to reformulate Connes and St{\o}rmer's transitivity theorem, \cite[Theorem 4]{CS76} in the following way. Let $M$ be a von Neumann algebra and $\vphi \in \St(M)$. We define the \emph{asymptotic centralizer} $M^{\om,\vphi}$ of $\vphi$ as the quotient
$$
M^{\om,\vphi} = \frac{\bigl\{x \in \ell^\infty(\N,M) \bigm| \lim_{n \to \om} \|x_n \vphi - \vphi x_n\| = 0 \bigr\}}{\bigl\{x \in \ell^\infty(\N,M) \bigm| \lim_{n \to \om} (\|x_n\|_\vphi + \|x_n^*\|_\vphi) = 0 \bigr\}} \; .
$$
Note that $M^{\om,\vphi}$ is a von Neumann algebra and that the formula $\vphi^\om(x) = \lim_{n \to \om} \vphi(x_n)$ defines a faithful normal tracial state on $M^{\om,\vphi}$.

Although we do not need this here, note that by \cite[Lemma 4.36]{AH12}, the von Neumann algebra $M^{\om,\vphi}$ coincides with the centralizer of the ultraproduct state $\vphi^\om$ on the Ocneanu ultrapower $M^\om$. Using that result, the following lemma is a consequence of \cite[Theorem 4.20 and Proposition 4.24]{AH12} and even holds without separability assumptions. We include a direct proof of this lemma that does not use the Ocneanu ultrapower but only the asymptotic centralizer. Still the argument is essentially the same as in \cite{AH12}.

\begin{lemma}\label{lem.asympt-centr-factor}
Let $M$ be a type III$_1$ factor with separable predual and $\vphi \in \St(M)$. Then $M^{\om,\vphi}$ is a II$_1$ factor.
\end{lemma}
\begin{proof}
Let $p,q \in M^{\om,\vphi}$ be nonzero projections with $\vphi^\om(p) = \lambda = \vphi^\om(q)$. We prove that $p$ and $q$ are equivalent projections in $M^{\om,\vphi}$. We can represent $p$ and $q$ by sequences of projections $p_n,q_n \in M$ such that $\vphi(p_n) = \lambda = \vphi(q_n)$ for all $n \in \N$ and such that $\lim_{n \to \om} \|p_n \vphi - \vphi p_n\| = 0 = \lim_{n \to \om} \|q_n \vphi - \vphi q_n\|$. Since $M$ is a type III factor, we can choose $v_n \in M$ such that $v_n v_n^* = p_n$ and $v_n^* v_n = q_n$ for all $n$.

Write $\vphi_n = p_n \vphi p_n$ and $\psi_n = v_n \vphi v_n^*$. Viewing $\lambda^{-1} \vphi_n$ and $\lambda^{-1} \psi_n$ as faithful normal states on the type III$_1$ factor $p_n M p_n$, it follows from Connes and St{\o}rmer's \cite[Theorem 4]{CS76} that we can choose $u_n \in \cU(p_n M p_n)$ such that $\|u_n \psi_n u_n^* - \vphi_n\| \to 0$. Writing $w_n = u_n v_n$, the sequence $(w_n)_{n \in \N}$ defines an element $w \in M^{\om,\vphi}$ satisfying $ww^* = p$ and $w^* w = q$.

Let $s \in [0,1]$. We next prove that there exists a projection $P \in M^{\om,\vphi}$ such that $\vphi^\om(P) = s$. In combination with the fact that $\vphi^\om$ is a faithful normal tracial state and the fact that projections having the same trace are equivalent, it then follows that $M^{\om,\vphi}$ is a factor of type II$_1$. Choose a projection $p \in M$ with $\vphi(p) = s$. Write $q = 1-p$ and define $\psi \in \St(M)$ by $\psi = p\vphi p + q \vphi q$. By Connes and St{\o}rmer's \cite[Theorem 4]{CS76}, we can choose $u_n \in \cU(M)$ such that $\|u_n \psi u_n^* - \vphi\| \to 0$. Since $p \in M_\psi$, it follows that $P = (u_n p u_n^*)_{n \in \N}$ defines a projection in $M^{\om,\vphi}$ with $\vphi^\om(P) = \psi(p) = \vphi(p) = s$.
\end{proof}

For every von Neumann algebra $M$ and every $\vphi \in \St(M)$, we denote by $E_\vphi$ the unique $\vphi$-preserving conditional expectation from $M$ onto $M_\vphi$.

\begin{lemma} \label{lem.semicontinuous}
Let $M$ be a von Neumann algebra. Take $x \in M$. Then the map
$$ \St(M) \ni \vphi \mapsto \| E_\vphi(x)\|_\vphi \in \R_+ $$ is upper semicontinuous.
\end{lemma}
\begin{proof}
Fix $x \in M$. Since $\|E_\vphi(x)\|_\vphi^2 = \|x\|_\vphi^2 - \|x - E_\vphi(x)\|_\vphi^2$ and since $\vphi \mapsto \|x\|_\vphi$ is continuous, it suffices to prove that $\vphi \mapsto \|x - E_\vphi(x)\|_\vphi$ is lower semicontinuous. Assume the contrary. Take $\delta \geq 0$ and a sequence $(\vphi_n)_{n \in \N}$ in $\St(M)$ that converges to $\vphi \in \St(M)$ such that
$$\|x - E_{\vphi_n}(x)\|_{\vphi_n} \leq \delta < \|x - E_\vphi(x)\|_\vphi \quad\text{for all $n \in \N$.}$$
After passage to a subsequence, we may assume that $E_{\vphi_n}(x) \to y \in M$ weakly$^*$. Since for every $a \in M$, we have
$$\vphi(a y) = \lim_n \vphi(a E_{\vphi_n}(x)) = \lim_n \vphi_n(a E_{\vphi_n}(x)) = \lim_n \vphi_n(E_{\vphi_n}(x)a) = \lim_n \vphi(E_{\vphi_n}(x)a) = \vphi(ya) \; ,$$
we get that $y \in M_\vphi$. Since $x - E_{\vphi_n}(x) \to x-y$ weakly$^*$ and since $y \in M_\vphi$, we arrive at the contradiction
\begin{equation*}
\|x-y\|_\vphi \leq \limsup_n \|x - E_{\vphi_n}(x)\|_\vphi \leq \delta < \|x - E_\vphi(x)\|_\vphi \leq \|x - y \|_\vphi \; .\qedhere
\end{equation*}
\end{proof}

The next lemma is the key step in proving Theorem~\ref{thmA}.

\begin{lemma}\label{lem.open-dense}
Let $M$ be a type III$_1$ factor with separable predual. Take $x \in M$ and $\varepsilon > 0$. Then
$$U(x,\varepsilon) = \bigl\{ \vphi \in \St(M) \bigm|  \| E_\vphi(x)\|_\vphi < |\vphi(x) | + \varepsilon \bigr\}$$
is a dense open subset of $\St(M)$.
\end{lemma}
\begin{proof}
By Lemma \ref{lem.semicontinuous}, the map
$$\vphi \mapsto \| E_\vphi(x) \|_\vphi -|\vphi(x)| $$
is upper semicontinuous, so that $U(x,\eps)$ is open. We fix an arbitrary $\eps > 0$ and prove that $U(x,7\eps)$ is dense in $\St(M)$.

Fix $\psi \in \St(M)$ and $\delta_0 > 0$. We construct a $\vphi \in U(x,7\eps)$ with $\|\psi - \vphi\| < \delta_0$. By Lemma \ref{lem.semicontinuous}, we find $\delta > 0$ such that $\| E_\vphi(x-E_\psi(x)) \|_\vphi < \varepsilon $ for every $\vphi \in \St(M)$ satisfying $\| \vphi-\psi \| \leq \delta$. We may assume that $\delta \leq 1$, $\delta \|x\| \leq \varepsilon$ and $\delta < \delta_0$.

Write $y=E_\psi(x)-\psi(x)1$. By Lemma \ref{lem.asympt-centr-factor}, $N=M^{\omega,\psi}$ is a II$_1$ factor with trace $\psi^\om$. Since $y \in N$ and $\psi^\om(y)=0$, Lemma \ref{lem.local-quantization} provides a finite dimensional abelian subalgebra $A \subset N$ such that $\| E_{A' \cap N}(y)\|_2 \leq \varepsilon $.

Denote by $P_1,\ldots,P_n$ the minimal projections in $A$, so that $E_{A' \cap N}(y) = \sum_{i=1}^n P_i y P_i$. We represent $P_i = (p_{i,k})_{k \in \N}$, where for each $k \in \N$, $(p_{1,k},\ldots,p_{n,k})$ is a partition of unity in $M$. Taking $k$ large enough, we thus find a partition of unity $(p_1,\dots,p_n)$ in $M$ such that
$$\bigl\| \sum_{i=1}^n p_i y p_i \bigr\|_\psi \leq 2\varepsilon \quad\text{and}\quad \bigl\| \psi- \sum_{i=1}^n p_i \psi p_i \bigr\| \leq \delta/2 \; .$$
Let $\phi =  \sum_{i=1}^n p_i \psi p_i$ and $B=\oplus_{i=1}^n \C p_i \subset M_\phi$. Since $M$ has separable predual, the point spectrum $\Lambda$ of $\phi$ is countable. Therefore, we can find a tuple of positive real numbers $(\lambda_i)_{1 \leq i \leq n}$ such that
\begin{itemlist}
\item $\sum_{i=1}^n \lambda_i \phi(p_i)=1$,
\item $|\lambda_i-1| \leq \delta/2$ for all $i$,
\item $\lambda_i \lambda_j^{-1} \not\in \Lambda$ for every $i \neq j$.
\end{itemlist}
Then the state $\vphi = \sum_{i=1}^n \lambda_i p_i \psi p_i$ will satisfy $\vphi \leq (1+\delta/2)\phi \leq  2 \phi$ and $M_\vphi \subset B' \cap M$. Therefore
$$ \| E_{\vphi}(y)  \|_\vphi \leq \| E_{B' \cap M}(y)\|_\vphi= \bigl\| \sum_{i=1}^n p_i y p_i \bigr\|_\vphi \leq 2 \bigl\| \sum_{i=1}^n p_i y p_i \bigr\|_\phi = 2 \bigl\| \sum_{i=1}^n p_i y p_i \bigr\|_\psi \leq  4\varepsilon \; .$$
Since $y=E_\psi(x)-\psi(x)1$, we thus get
$$ \| E_\vphi(E_\psi(x)) \|_\vphi \leq | \psi(x)|+ \| E_\vphi(E_\psi(x)) - \psi(x)1 \| =| \psi(x)|+\| E_\vphi(y) \|_\vphi \leq |\psi(x)| + 4 \varepsilon \; .$$

Note that $\| \vphi - \psi \| \leq \| \vphi - \phi\| + \| \phi-\psi\| \leq \delta$.  Therefore, by our choice of $\delta$, we have
$$\| E_\vphi(x)-E_\vphi(E_\psi(x)) \|_\vphi \leq \varepsilon \; .$$
Moreover, we have
$$| \psi(x)| \leq | \vphi(x)| + \delta \|x\| \leq |\vphi(x)|+ \varepsilon \; .$$
We conclude that
$$ \| E_\vphi(x) \|_\vphi \leq | \vphi(x)|+6\varepsilon \; .$$
So, $\vphi \in U(x,7\eps)$ and $\|\psi-\vphi\| \leq \delta < \delta_0$.
\end{proof}

We will deduce Theorem~\ref{thmA} from Lemma \ref{lem.open-dense} and the Baire category theorem. Since $\St(M)$ is not closed in the Polish space $\St_1(M)$, we also need the following lemma.

\begin{lemma}\label{lem.faithful-is-dense-Gdelta}
Let $M$ be a von Neumann algebra with separable predual. Then $\St(M) \subset \St_1(M)$ is a dense $G_\delta$ set.
\end{lemma}
\begin{proof}
Choose a faithful normal state $\vphi$ on $M$. For every $\psi \in \St_1(M)$, denote by $\supp \psi$ its support projection, which is the smallest projection $p \in M$ such that $\psi(1-p) = 0$. For every $\eps > 0$, define $V(\eps) \subset \St_1(M)$ by
$$V(\eps) = \{\psi \in \St_1(M) \mid \vphi(1-\supp \psi) < \eps \} \; .$$
We claim that each $V(\eps)$ is an open dense subset of $\St_1(M)$ and that $\bigcap_{k=1}^\infty V(1/k) = \St(M)$, which then concludes the proof of the lemma.

To prove that $V(\eps)$ is open, let $\psi_n \in \St_1(M) \setminus V(\eps)$ be a sequence that converges to $\psi \in \St_1(M)$. We have to prove that $\psi \not\in V(\eps)$. Denote $p_n = 1 - \supp \psi_n$ and $p = 1 - \supp \psi$. We have to prove that $\vphi(p) \geq \eps$.

Choose a subsequence $(n_k)_k$ such that $p_{n_k} \to a \in M$ weakly$^*$. Note that $0 \leq a \leq 1$. We get that
$$\psi(a) = \lim_k \psi(p_{n_k}) = \lim_k \psi_{n_k}(p_{n_k}) = 0 \; .$$
This means that $a = p a p$. Since $0 \leq a \leq 1$, it follows that $a \leq p$, so that $\vphi(p) \geq \vphi(a)$. Since $p_{n_k} \to a$ weakly$^*$, also $\vphi(p_{n_k}) \to \vphi(a)$. Since $\vphi(p_{n_k}) \geq \eps$ for all $k$, we conclude that $\vphi(a) \geq \eps$. Hence, $\vphi(p) \geq \eps$ and we have proven that $V(\eps)$ is an open subset of $\St_1(M)$.

By definition, $\bigcap_{k=1}^\infty V(1/k) = \St(M)$. Since $\St(M)$ is dense in $\St_1(M)$, in particular all $V(\eps)$ are dense. So the lemma is proven.
\end{proof}

We are now ready to prove Theorem~\ref{thmA}.

\begin{proof}[{Proof of Theorem \ref{thmA}}]
The implication $3 \Rightarrow 2$ is trivial and the implication $2 \Rightarrow 1$ follows from Lemma \ref{lem.easy-implication}. Now assume that $1$ holds.

Observe that for every $x \in M$ and every $\vphi \in \St(M)$, we have
$$ \| E_\vphi(x) \|_\vphi^2=\| E_\vphi(x)-\vphi(x)1 \|_\vphi^2+|\vphi(x)|^2$$
hence $E_\vphi(x)=\vphi(x)1$ if and only if $\| E_\vphi(x)\|_\vphi \leq |\vphi(x)|$. From this, we see that
$$\Serg(M) = \bigl\{ \vphi \in \St(M) \bigm| \forall x \in M, \;  E_\vphi(x)=\vphi(x)1 \bigr\} = \bigcap_{x \in M, \: \varepsilon > 0} U(x,\varepsilon)$$
where the sets $U(x,\varepsilon)$ are the dense open subsets of $\St(M)$ given by Lemma \ref{lem.open-dense}.

To write this as a countable intersection, we pick $(x_n)_{n \in \N}$ a strongly dense sequence of elements in the unit ball of $M$ and a sequence $(\varepsilon_n)_{n \in \N}$ of positive numbers that converges to $0$. Then we have
$$\Serg(M) = \bigcap_{n,k \in \N} U(x_n, \varepsilon_k) \; .$$
In combination with Lemma \ref{lem.faithful-is-dense-Gdelta} and the Baire category theorem, we conclude that $\Serg(M)$ is a dense $G_\delta$ subset of $\St_1(M)$.
\end{proof}

\begin{remark}\label{rem.not-separable}
We note that Theorem~\ref{thmA} is false without the separability assumption on the predual $M_*$. When $M$ is any type III$_1$ factor with separable predual, $\vphi \in \St(M)$ and $\om$ is a free ultrafilter on $\N$, it follows from \cite[Theorem 4.20]{AH12} that the Ocneanu ultrapower $M^\om$ is a factor of type III$_1$ with the property that all its faithful normal states are unitarily conjugate. For the ultrapower state $\vphi^\om$, the centralizer is diffuse by Lemma \ref{lem.asympt-centr-factor}. Since all other faithful normal states on $M^\om$ are unitarily conjugate to $\vphi^\om$, they all have a diffuse centralizer.
\end{remark}

\section{\boldmath Ergodic actions on type II$_1$ factors}\label{sec.ergodic-actions}

We will need the following standard averaging lemma.

\begin{lemma}[{See e.g.\ \cite[Lemma 5.1]{AHHM18}}] \label{lem.average}
Let $M$ be a von Neumann algebra and $\vphi \in \St(M)$ a faithful normal state. Let $G \subset \Aut(M,\vphi)$ be a group of $\vphi$-preserving automorphisms. Denote by $E: M \to M^{G}$ the unique $\vphi$-preserving conditional expectation on the fixed point algebra $M^G$. Whenever $x \in M$ and $x \not\in M^G$, there exists an $\alpha \in G$ such that $\|x-\alpha(x)\|_\vphi > \|x - E(x)\|_\vphi$.
\end{lemma}

\begin{proof}[{Proof of Theorem \ref{thmB}}]
For every nonzero element $x \in M$ with $\tau(x) = 0$, write
$$\cU_x = \bigl\{ v \in \cC(\al) \bigm| \; \exists h \in \Gamma : \|v_h \al_h(x) v_h^* - x\|_2 > \|x\|_2 \;\bigr\} \; .$$
We first prove that $\cU_x$ is an open dense subset of $\cC(\al)$. It is clear that $\cU_x$ is open. To prove that every $w \in \cC(\al)$ belongs to the closure of $\cU_x$, we replace $\al$ by $\Ad w \circ \al$ and it suffices to prove that $1$ lies in the closure of $\cU_x$.

Fix $\eps > 0$ and a finite subset $\cF \subset \Gamma$. We construct $v \in \cU_x$ such that $\|v_g - 1\|_2 < \eps$ for all $g \in \cF$. We will find $v$ as an inner $1$-cocycle $v_h = u^* \al_h(u)$ for all $h \in \Gamma$. It thus suffices to prove that we can find a unitary $u \in \cU(M)$ and an element $h \in \Gamma$ such that $\|\al_g(u) - u\|_2 < \eps$ for all $g \in \cF$ and $\|\al_h(u x u^*) - u x u^*\|_2 > \|x\|_2$.

Fix a free ultrafilter $\om$ on $\N$. Denote by $\al^\om$ the natural action of $\Gamma$ on $M^\om$. By \cite[Lemma 4.2]{PSV18}, we can take a unitary $V \in M^\om$ such that the subfactors $M$ and $(\al^\om_h(V M V^*))_{h \in \Gamma}$ are all free inside $M^\om$. Denote by $P \subset M^\om$ the subfactor generated by $(\al^\om_h(V M V^*))_{h \in \Gamma}$. By \cite[Lemma 4.3]{PSV18}, the restriction of $\al^\om$ to $P$ is not strongly ergodic, and its ultrapower actually has a diffuse fixed point subalgebra. We can thus choose a projection $p \in P$ such that $\tau(p) = 1/2$ and $\|\al^\om_g(p) - p\|_2 < \eps / 6$ for all $g \in \cF$.

For every finite subset $J \subset \Gamma$, we denote by $P_J \subset P$ the subfactor generated by $\al^\om_h(V M V^*)$, $h \in J$. Since the union of all $P_J$ is dense in $P$ and every $P_J$ is a II$_1$ factor, we can pick a finite subset $J \subset \Gamma$ and a projection $q \in P_J$ such that $\tau(q) = 1/2$ and $\|p-q\|_2 < \eps/6$. Since $\|\al^\om_g(p) - p\|_2 < \eps / 6$ for all $g \in \cF$, it follows that $\|\al^\om_g(q) - q\|_2 < \eps/2$ for all $g \in \cF$. Defining $U = 2q-1$, we have found a unitary $U \in P_J$ with $\tau(U)=0$ and $\|\al^\om_g(U) - U\|_2 < \eps$ for all $g \in \cF$.

Since $\Gamma$ is infinite, we can take $h \in \Gamma$ such that $h J \cap J = \emptyset$. Note that $\al^\om_h(U) \in P_{hJ}$. Since $x \in M$ and $U \in P_J$ have trace zero and $M$, $P_J$, $P_{hJ}$ are free, we have
$$\| \al^\om_h(U x U^*) - U x U^*\|_2^2 = 2 \bigl( \|x\|_2^2 - \Re \tau\bigl(U x^* U^* \al^\om_h(U) \al_h(x) \al^\om_h(U^*)\bigr)\bigr) = 2 \|x\|_2^2 > \|x\|_2^2 \; .$$
Representing $U = (u_n)_{n \in \N}$ with $u_n \in \cU(M)$, there exists an $n \in \N$ such that
$$\|\al_g(u_n) - u_n\|_2 < \eps \quad\text{for all $g \in \cF$, and}\quad \| \al_h(u_n x u_n^*) - u_n x u_n^*\|_2 > \|x\|_2 \; .$$
Putting $u = u_n$, the density of $\cU_x$ is proven.

Let $(x_n)_{n \in \N}$ be a countable $\|\cdot\|_2$-dense subset of $\{x \in M \mid x \neq 0 , \tau(x) = 0\}$. To conclude the proof of the theorem, it suffices to show that for $v \in \cC(\al)$, the action $\Ad v \circ \al$ is ergodic if and only if $v \in \bigcap_{n \in \N} \cU_{x_n}$.

If $\Ad v \circ \al$ is ergodic, then by Lemma \ref{lem.average}, $v \in \cU_x$ for every nonzero $x \in M$ with $\tau(x) = 0$. So, $v \in \bigcap_{n \in \N} \cU_{x_n}$. If $\be = \Ad v \circ \al$ is not ergodic, we choose a nonzero element $x \in M$ with $\tau(x) = 0$ and $\be_h(x) = x$ for all $h \in \Gamma$. Write $\eps = \|x\|_2 / 3$. Take $n \in \N$ such that $\|x - x_n\|_2 \leq \eps$. Then $\|x_n\|_2 \geq \|x\|_2 - \eps = 2 \eps$ and
$$\|\be_h(x_n) - x_n\|_2 \leq 2 \eps \leq \|x_n\|_2 \quad\text{for all $h \in \Gamma$.}$$
This says that $v \not\in \cU_{x_n}$. Thus, $v \not\in \bigcap_{n \in \N} \cU_{x_n}$.
\end{proof}

\begin{theorem}\label{thm.ergodic-action-Z}
Let $M$ be a II$_1$ factor with separable predual and $\al \in \Aut(M)$. Then the following statements are equivalent.
\begin{enumlist}
\item $\al \in \Out(M) = \Aut(M) / \Inn(M)$ has infinite order.
\item There exists a unitary $u \in \cU(M)$ such that $\Ad u \circ \al$ is ergodic.
\item The set $\{u \in \cU(M) \mid \Ad u \circ \al \;\;\text{is ergodic}\;\}$ is a dense $G_\delta$ subset of $\cU(M)$.
\end{enumlist}
\end{theorem}
\begin{proof}
$1 \Rightarrow 3$ follows from Theorem~\ref{thmB} and $3 \Rightarrow 2$ is trivial. To conclude the proof, we assume that $\al^n$ is inner for some $n \in \N$, $n \geq 1$, and prove that no $\Ad v \circ \al$ is ergodic. Replacing $\al$ by $\Ad v \circ \al$, we still have that $\al^n$ is inner and we have to prove that $\al$ is not ergodic. Write $\al^n = \Ad u$ for some $u \in \cU(M)$. Define $P = M \cap \{u\}'$ and note that $P$ is diffuse. Also note that
$$\Ad u = \al^n = \al \circ \al^n \circ \al^{-1} = \al \circ \Ad u \circ \al^{-1} = \Ad \al(u) \; .$$
Thus, $\al(u)$ is a multiple of $u$, so that $\al(P) = P$. Since $\al^n$ is the identity on $P$, $\al$ defines an action of the finite group $\Z/n\Z$ on the diffuse von Neumann algebra $P$. This action is not ergodic, so that also $\al$ is not ergodic.
\end{proof}

\begin{definition}
We denote by $\Erg$ the class of countable groups $\Gamma$ with the property that every outer action of $\Gamma$ on a II$_1$ factor with separable predual admits a cocycle perturbation that is ergodic.
\end{definition}

By Theorem~\ref{thmB}, the class $\Erg$ contains all infinite amenable groups. The class $\Erg$ also has the following stability property, similar to \cite[Corollary 2.3]{Pop18} for the vanishing $2$-cohomology property. In particular, the free groups belong to $\Erg$.

\begin{proposition}\label{prop.free-product}
If $G = \Gamma \ast_T \Lambda$ is any amalgamated free product with $T$ finite, $\Lambda$ any countable group and $\Gamma \in \Erg$, then $G \in \Erg$.
\end{proposition}
\begin{proof}
Let $G \actson^\al M$ be an outer action on a II$_1$ factor $M$ with separable predual. Denote by $\be$ the restriction of $\al$ to the subgroup $\Gamma$. Since $\Gamma \in \Erg$, we can take $v \in \cC(\be)$ such that $\Ad v \circ \be$ is ergodic. By \cite[Theorem 3.1.3]{Jon79}, the restriction of $v$ to the finite group $T$ is a coboundary. We thus find a unitary $u \in \cU(M)$ such that $v = u^* \be_k(u)$ for all $k \in T$. Define $w \in \cC(\be)$ by $w_g = u v_g \be_g(u^*)$ for all $g \in \Gamma$. Then $\Ad w \circ \be$ is conjugate to $\Ad v \circ \be$ via the inner automorphism $\Ad u$. Therefore, also $\Ad w \circ \be$ is ergodic.

By construction, $w_k = 1$ for all $k \in T$. We thus find a unique $c \in \cC(\al)$ such that $c_g = 1$ for all $g \in \Lambda$ and $c_g = w_g$ for all $g \in \Gamma$. Since $\Ad w \circ \be$ is ergodic, a fortiori $\Ad c \circ \al$ is ergodic.
\end{proof}

On the other hand, as with the vanishing $2$-cohomology in \cite{Pop18}, there also are several families of groups that do not belong to $\Erg$. The result follows from a variant of Popa's cocycle superrigidity theorems \cite{Pop01,Pop06}.

\begin{proposition}\label{prop.no-go-thm}
If $\Gamma$ has property (T) or if $\Gamma = \Gamma_1 \times \Gamma_2$ is the direct product of a nonamenable group $\Gamma_1$ and an infinite group $\Gamma_2$, then $\Gamma$ admits an outer action $\al$ on the hyperfinite II$_1$ factor such that for no $v \in \cC(\al)$, the action $\Ad v \circ \al$ is ergodic. In particular, $\Gamma \not\in \Erg$.
\end{proposition}
\begin{proof}
Let $R_0$ be a copy of the hyperfinite II$_1$ factor and choose an outer action $\Z \actson R_0$. Write $R_1 = R_0^{\ot \Gamma}$ and consider the diagonal action $\Z \actson R_1$. Then write $R = R_1 \rtimes \Z$ and define $\Gamma \actson^\al R$ by taking the Bernoulli action $\Gamma \actson R_1$ and the trivial action of $\Gamma$ on $L(\Z) \subset R$.

Let $v \in \cC(\al)$ and $\be = \Ad v \circ \al$. We have to prove that $\be$ is not ergodic. A finite group action on a II$_1$ factor is never ergodic, so that we may assume that $\Gamma$ is an infinite property~(T) group or the direct product of a nonamenable group and an infinite group. In both cases, Popa's cocycle superrigidity theorems for the Bernoulli actions of $\Gamma$ (see \cite{Pop01,Pop06}) admit a generalization to crossed products and tensor products, as written in detail in \cite[Theorem 7.1]{VV14}. Writing $A = L(\Z) \subset R$, we find a projection $p \in M_n(\C) \ot A$, a partial isometry $V \in (M_{1,n}(\C) \ot R)p$ and a unitary representation $\pi : \Gamma \to \cU(p(M_n(\C) \ot A)p)$ such that $V^* V = p$, $VV^*$ is $\be$-invariant and $v_g (\id \ot \al_g)(V) = V \pi(g)$ for all $g \in \Gamma$. Then $V (1 \ot A) V^*$ is a diffuse subalgebra of a corner of $M^\be$, so that $\be$ is not ergodic.
\end{proof}

We conclude this section by proving that the separability assumption in Theorem~\ref{thmB} is essential.

\begin{remark}\label{rem.non-separable-actions}
Let $I$ be an uncountable set and define $M$ as the $I$-fold tensor product of $(M_2(\C),\tau)$. Then every infinite countable group $\Gamma$ admits an outer action $\al$ on $M$ for which no cocycle perturbation is ergodic. To construct $\alpha$, identify $I$ with the disjoint union of $\Gamma$ and an uncountable set $J$. For every subset $K \subset I$, we have the corresponding subalgebra $M_K \subset M$ given by the $K$-fold tensor product. In particular, $M = M_\Gamma \ovt M_J$ and we define $\al_g = \be_g \ot \id$, where $\Gamma \actson^\be M_\Gamma$ is the usual Bernoulli action. Then $\al$ is an outer action. If $v \in \cC(\al)$, the countability of $\Gamma$ implies that we can find a countable subset $K \subset J$ such that $v_g \in M_\Gamma \ovt M_K$ for all $g \in \Gamma$. But then $1 \ot M_{J \setminus K}$ belongs to the fixed point algebra of $\Ad v \circ \al$. So $\al$ has no ergodic cocycle perturbation.
\end{remark}

\section{Concluding remarks and open problems}\label{sec.problems}
By Lemma \ref{lem.easy-implication}, we know that if $\vphi$ is an ergodic state on some factor $M$ then the unitary representation $(\sigma_t^\vphi)_{t \in \R}$ on $L^2(M,\vphi) \ominus \C$ is weakly mixing. It is thus natural to ask if Theorem~\ref{thmA} can be strengthened in the following way.

\begin{problem} Let $M$ be a type III$_1$ factor with separable predual. Can we find $\vphi \in \Serg(M)$ such that the unitary representation $(\sigma_t^\vphi)_{t \in \R}$ on $L^2(M,\vphi) \ominus \C$ is mixing, or even a multiple of the regular representation?
\end{problem}

Note that examples of such mixing states are given by the quasi-free states on Araki-Woods factors (or free-quasi-free states on free Araki-Woods factors) associated to mixing orthogonal representations of $\R$ by the CAR functor.

The next problem asks for a relative version of Theorem~\ref{thmA}.

\begin{problem}\label{problem.inclusion}
Let $N \subset M$ be an inclusion of von Neumann algebras with faithful normal conditional expectation $E : M \to N$. When does there exist an ergodic $\vphi \in \St(M)$ that is $E$-invariant: $\vphi \circ E = \vphi$~?
\end{problem}

In Theorem \ref{thm.ergodic-inclusion}, we provide a criterion and a necessary condition towards answering Problem \ref{problem.inclusion}. In Remark \ref{rem.bicentralizer}, we then relate Problem \ref{problem.inclusion} to Connes' bicentralizer problem. Before presenting this, we need some notation and preliminaries.

Recall that $\St_1(M)$ denotes the set of normal states on $M$. We define $\St_1(M,E) = \{\vphi \in \St_1(M) \mid \vphi \circ E = \vphi\}$, $\St(M,E) = \St_1(M,E) \cap \St(M)$ and $\Serg(M,E) = \St(M,E) \cap \Serg(M)$. We first exclude the trivial case $N = \C 1$, because then $\St(M,E)$ is a singleton. Since any $\vphi \in \Serg(M,E)$ is in particular an ergodic state on $N$, the set $\Serg(M,E)$ can only be nonempty if $N$ is a factor of type III$_1$. So we only consider Problem \ref{problem.inclusion} when $N$ is a III$_1$ factor.

Given a type III$_1$ factor $N$ and an inclusion $N \subset M$ with expectation $E : M \to N$, for every $\vphi \in \St(M,E)$, the \emph{relative bicentralizer} $B(N\subset M,\vphi)$ was studied in \cite{AHHM18}. It is defined as the von Neumann subalgebra of elements $x \in M$ that satisfy $a_n x - x a_n \to 0$ $*$-strongly for every bounded sequence $a_n \in N$ with $\|a_n \vphi - \vphi a_n\| \to 0$.

By \cite[Theorem A]{AHHM18}, the relative bicentralizer $B(N \subset M,\vphi)$ does not depend on the choice of $\vphi \in \St(M,E)$ up to the following canonical isomorphism. When $\psi \in \St(M,E)$, by Connes and St{\o}rmer's \cite[Theorem 4]{CS76}, we can choose a sequence of unitaries $u_n \in \cU(N)$ such that $\|u_n\vphi u_n^* - \psi \| \to 0$. By \cite[Theorem A]{AHHM18}, there is a unique $*$-isomorphism $\be$ from $B(N\subset M,\vphi)$ to $B(N \subset M,\psi)$ satisfying $u_nxu_n^* \to \be(x) $ $*$-strongly for all $x \in B(N \subset M,\vphi)$.

Note that the modular automorphism group $\si^\vphi$ leaves $B(N \subset M,\vphi)$ globally invariant. Therefore, the modular automorphism group of $\vphi|_{ B(N \subset M,\vphi)}$ is given by the restriction $\si^\vphi|_{B(N \subset M,\vphi)}$. Since $\|u_n\vphi u_n^* - \psi \| \to 0$, the isomorphism $\beta$ must be state-preserving and $\be$ intertwines $\si^\vphi$ and $\si^\psi$.

Finally recall that $E : M \to N$ gives rise to a canonical embedding $c(N) \subset c(M)$ of the continuous cores. The following is a version of Theorem~\ref{thmA} for inclusions.

\begin{theorem}\label{thm.ergodic-inclusion}
Let $N$ be a factor of type III$_1$ and $N \subset M$ an inclusion with faithful normal conditional expectation $E : M \to N$. Assume that $M_*$ is separable. Let $\om$ be a free ultrafilter on $\N$. Consider the following statements.
\begin{enumlist}
\item $M$ admits a faithful normal state that is ergodic and $E$-invariant, i.e.\ $\Serg(M,E) \neq \emptyset$.
\item $\Serg(M,E)$ is a dense $G_\delta$ subset of $\St_1(M,E)$.
\item $B(N \subset M,\vphi)^{\si^\vphi} = \C 1$ for some (or equivalently all) $\vphi \in \St(M,E)$.
\item $c(N) \subset c(M)$ is irreducible.
\end{enumlist}
Then statements 1, 2 and 3 are equivalent and they imply statement 4.
\end{theorem}

\begin{remark}\label{rem.bicentralizer}
In \cite{Mar23}, it is conjectured that $4 \Rightarrow 3$, so that conjecturally all statements in Theorem \ref{thm.ergodic-inclusion} are equivalent. Moreover in \cite[Corollary 8.16]{Mar23}, it is proven that this conjecture is true whenever $N$ itself has trivial bicentralizer. The latter is known for large families of type III$_1$ factors, including the amenable type III$_1$ factor and all almost periodic type III$_1$ factors.
\end{remark}

\begin{remark}
A necessary condition for the irreducibility of $c(N) \subset c(M)$ is that $(N' \cap M)^{\si^\vphi} = \C 1$ for $\vphi \in \St(M,E)$. But this is far from being a sufficient condition. Indeed, given any type III$_1$ factor $N$ and given any ergodic action $\R \actson^\zeta P$ of $\R$ on a II$_1$ factor $P$, we consider the trace scaling action $\R \actson^\theta c(N)$. Since $\theta \ot \zeta$ is a trace scaling action of the II$_\infty$ factor $c(N) \ovt P$, we have that $M = (c(N) \ovt P)^{\theta \ot \zeta}$ is a III$_1$ factor with continuous core $c(M) = c(N) \ovt P$. Then $N \cong N \ot 1$ is a von Neumann subalgebra of $M$ and $\id \ot \tau$ defines a conditional expectation of $M$ onto $N$. Viewing $c(M)$ as $c(N) \ovt P$, we find that $c(N)' \cap c(M) \cong P$. Since $N \subset c(N)$ is irreducible and $\zeta$ is ergodic, also $N \subset M$ is irreducible, but $c(N) \subset c(M)$ is not irreducible.
\end{remark}

\begin{proof}[{Proof of Theorem \ref{thm.ergodic-inclusion}}]
We start by proving the nontrivial implication $3 \Rightarrow 2$. Note that by the discussion preceding the theorem, statement $3$ is independent of the choice of $\vphi \in \St(M,E)$. Fix $\vphi \in \St(M,E)$ and assume that $B(N \subset M,\vphi)^{\si^\vphi} = \C 1$. Consider the von Neumann algebra $M^{\om,\vphi}$ with its faithful normal tracial state $\vphi^\om$. We claim that
$$\begin{matrix} M_\vphi & \subset & M^{\om,\vphi} \\ \cup & & \cup \\ \C 1 & \subset & (N^{\om,\vphi})' \cap M^{\om,\vphi} \end{matrix}$$
is a commuting square. To prove this claim, take $x \in (N^{\om,\vphi})' \cap M^{\om,\vphi}$ and write $y = E_{M_\vphi}(x)$. We have to prove that $y \in \C 1$.

Represent $x$ by the bounded sequence $(x_n)_{n \in \N}$ satisfying $\lim_{n \to \om} \|x_n \vphi - \vphi x_n\| = 0$. Let $z = \lim_{n \to \om} x_n$ weakly. Since $x_n \vphi \to z \vphi$ and $\vphi x_n \to \vphi z$ weakly, it follows that $z \vphi = \vphi z$. So, $z \in M_\vphi$ and we conclude that $z = y$. By \cite[Proposition 3.3]{AHHM18}, $x$ belongs to the Ocneanu ultrapower of $B(N \subset M,\vphi)$. This implies that $z \in B(N \subset M,\vphi)$. Since $z \in M_\vphi$ and $B(N \subset M,\vphi)^{\si^\vphi} = \C 1$, it follows that $z \in \C 1$. So, $y \in \C 1$ and the claim is proven.

By the claim, for every $x \in M_\vphi$ with $\vphi(x) = 0$, we have that $E_{(N^{\om,\vphi})' \cap M^{\om,\vphi}}(x) = 0$. For every $\eps > 0$, Popa's local quantization \cite[Theorem A.1.2]{Pop92} provides a finite dimensional abelian von Neumann subalgebra $A \subset N^{\om,\vphi}$ such that $\|E_{A' \cap M^{\om,\vphi}}(x)\|_{\vphi^\om} \leq \eps$. We can then repeat the proof of Theorem~\ref{thmA} and obtain that $2$ holds.

The implications $2 \Rightarrow 1 \Rightarrow 3$ are trivial. To conclude the proof, we show $1 \Rightarrow 4$. Fix $\vphi \in \Serg(M,E)$ and write $c(M) = M \rtimes_{\sigma^\vphi} \R$. We have seen in the proof of Lemma \ref{lem.easy-implication} that $L(\R)' \cap c(M) = L(\R)$. Therefore, $c(N)' \cap c(M) \subset c(N)$, which forces $c(N) \subset c(M)$ to be irreducible, because $c(N)$ is a factor.
\end{proof}

\begin{problem}
Two potential generalizations of Theorem~\ref{thmB} are open and related to Theorem~\ref{thmA}.
\begin{itemlist}
\item Under which conditions, actions of amenable, \emph{locally compact, noncompact} groups on II$_1$ factors admit ergodic cocycle perturbations?
\item Under which conditions, amenable group actions on arbitrary factors admit ergodic cocycle perturbations?
\end{itemlist}
\end{problem}

Assume that $G$ is a locally compact abelian group and that $G \actson^\al M$ is an action on a factor $M$ such that for all $g \in G \setminus \{e\}$, the automorphism $\al_g$ is outer. A necessary condition for $\al$ to admit an ergodic cocycle perturbation is that the crossed product $M \rtimes_\al G$ is a factor. Indeed, if $v \in \cC(\al)$ is a $1$-cocycle such that $\be = \Ad v \circ \al$ is ergodic, we have $M \rtimes_\be G \cong M \rtimes_\al G$. Also, $L(G)' \cap M \rtimes_\be G \subset L(G)$, by the ergodicity of $\be$. Since the map $g \mapsto \be_g$ is faithful, it then follows that $M \rtimes_\be G$ is a factor.

On the other hand, the outerness of the individual automorphisms $\al_g$, $g \neq e$, is not sufficient to guarantee that $M \rtimes_\al G$ is a factor. For instance, an obvious modification of \cite[Example 8]{MV22} provides an action of $\R$ on the hyperfinite II$_1$ factor $M$ by outer automorphisms such that $M \rtimes \R$ is not a factor.

Recall that an action of a locally compact group $G$ on a factor $M$ is called \emph{strictly outer} if $M' \cap M \rtimes G = \C 1$. Then $M \rtimes G$ obviously is a factor. A more precise version of the first question above would therefore be: does every strictly outer action of an amenable, locally compact, noncompact group $G$ on a II$_1$ factor $M$ admit an ergodic cocycle perturbation?

The second question above seems to be even more difficult. While it is conceivable that a variant of Theorem~\ref{thmB} can be proven for \emph{state preserving} actions of countable groups, we have no idea what happens for general automorphisms and automorphism groups of arbitrary factors.


\begin{thebibliography}{AHHM18}\setlength{\itemsep}{-1mm} \setlength{\parsep}{0mm} \small

\bibitem[AH12]{AH12} H. Ando and U. Haagerup, Ultraproducts of von Neumann algebras. {\it J. Funct. Anal.} {\bf 266} (2014), 6842-6913.

\bibitem[AHHM18]{AHHM18} H. Ando, U. Haagerup, C. Houdayer and A. Marrakchi, Structure of bicentralizer algebras and inclusions of type III factors. {\it Math. Ann.} {\bf 376} (2020), 1145-1194.

\bibitem[Bau95]{Bau95} H. Baumg\"{a}rtel, Operator algebraic methods in quantum field theory. Akademie Verlag, Berlin, 1995.

\bibitem[Con72]{Con72} A. Connes, Une classification des facteurs de type III. {\it Ann. Sci. \'{E}cole Norm. Sup.} {\bf 6} (1973), 133-252.

\bibitem[CS76]{CS76} A. Connes and E. St{\o}rmer, Homogeneity of the state space of factors of type III$_1$. {\it J. Funct. Anal.} {\bf 28} (1978), 187-196.

\bibitem[HT70]{HT70} R.H. Herman and M. Takesaki, States and automorphism groups of operator algebras. {\it Comm. Math. Phys.} {\bf 19} (1970), 142-160.

\bibitem[Jon79]{Jon79} V. Jones, Actions of finite groups on the hyperfinite type II$_1$ factor. {\it Mem. Amer. Math. Soc.} {\bf 28} (1980), no.\ 237.

\bibitem[Lon78]{Lon78} R. Longo, Notes on algebraic invariants for noncommutative dynamical systems. {\it Comm. Math. Phys.} {\bf 69} (1979), 195-207.

\bibitem[Mar23]{Mar23} A. Marrakchi, Kadison's problem for type III subfactors and the bicentralizer conjecture. {\it Preprint.} \href{https://arxiv.org/abs/2308.15163}{arXiv:2308.15163}

\bibitem[MV22]{MV22} A. Marrakchi and S. Vaes, Spectral gap and strict outerness for actions of locally compact groups on full factors. {\it Enseign. Math.} {\bf 69} (2023), 353-379.

\bibitem[MU12]{MU12} M. Mart\'{\i}n and Y. Ueda, On the geometry of von Neumann algebra preduals. {\it Positivity} {\bf 18} (2014), 519-530.

\bibitem[Pop92]{Pop92} S. Popa, Classification of amenable subfactors of type II. {\it Acta Math.} {\bf 172} (1994), 163-255.

\bibitem[Pop01]{Pop01} S. Popa, Some rigidity results for non-commutative Bernoulli shifts. {\it J. Funct. Anal.} {\bf 230} (2006), 273-328.

\bibitem[Pop06]{Pop06} S. Popa, On the superrigidity of malleable actions with spectral gap. {\it J. Amer. Math. Soc.} {\bf 21} (2008), 981-1000.

\bibitem[Pop18]{Pop18} S. Popa, On the vanishing cohomology problem for cocycle actions of groups on II$_1$ factors. {\it Ann. Sci. \'{E}c. Norm. Sup\'{e}r.} {\bf 54} (2021), 407-443.

\bibitem[PSV18]{PSV18} S. Popa, D. Shlyakhtenko, Dimitri and S. Vaes, Classification of regular subalgebras of the hyperfinite II$_1$ factor. {\it J. Math. Pures Appl.} {\bf 140} (2020), 280-308.

\bibitem[Yng04]{Yng04} J. Yngvason, The role of type III factors in quantum field theory. {\it Rep. Math. Phys.} {\bf 55} (2005), 135-147.

\bibitem[VV14]{VV14} S. Vaes and P. Verraedt, Classification of type III Bernoulli crossed products. {\it Adv. Math.} {\bf 281} (2015), 296-332.
\end{thebibliography}
\end{document}